\documentclass[final,leqno]{siamltex704}
\usepackage{amsmath}
\usepackage{amssymb}
\usepackage{graphicx}
\usepackage{subfig}
\usepackage{bm}
\usepackage{booktabs}
\usepackage{multirow}
\usepackage[notcite,notref]{showkeys}
\usepackage[numbers,sort&compress]{natbib}
\usepackage[figuresright]{rotating}
\usepackage{verbatim}
\usepackage{color,xcolor}
\usepackage{tikz}
\usepackage{xcolor}
\usepackage{eso-pic}
\newcommand{\watermark}[3]{\AddToShipoutPictureBG{
\parbox[b][\paperheight]{\paperwidth}{
\vfill%
\centering%
\tikz[remember picture, overlay]%
  \node [rotate = #1, scale = #2] at (current page.center)%
    {\textcolor{gray!80!cyan!30}{#3}};
\vfill}}}
\newcommand{\bn}{{\bm n}}
\newcommand{\bx}{{\bm x}}

\newcommand{\pK}{{\partial K}}

\def\T{{\mathcal T}}
\def\E{{\mathcal E}}

\def\l{{\langle}}
\def\r{{\rangle}}
\def\3bar{{|\hspace{-.02in}|\hspace{-.02in}|}}
\def\ljump{{[\hspace{-0.02in}[}}
\def\rjump{{]\hspace{-0.02in}]}}
\def\la{{\{\hspace{-0.045in}\{}}
\def\ra{{\}\hspace{-0.045in}\}}}

\newtheorem{remark}{Remark}[section]
\newtheorem{example}{\bf Example}
\newtheorem{algorithm}{Method}

\title{A stabilizer-free $C^0$ weak Galerkin method for the biharmonic equations}

\author{Peng Zhu\thanks{College of
Data Science, Jiaxing University, Jiaxing, Zhejiang 314001, China
(zhupeng.hnu@gmail.com). This research was supported in part by Zhejiang Provincial Natural Science Foundation of China (LY19A010008) and Natural Science Foundation of China (12071184).}
\and
Shenglan Xie\thanks{College of Information Engineering, Jiaxing Nanhu University, Jiaxing, Zhejiang 314001, China (shlxie@126.com).}
\and
Xiaoshen Wang\thanks{Department of
Mathematics, University of Arkansas at Little Rock, Little Rock, AR
72204 (xxwang@ualr.edu).}
}

\begin{document}

\maketitle

\begin{abstract}
In this article, we present and analyze a stabilizer-free $C^0$ weak Galerkin (SF-C0WG) method for solving the biharmonic problem. The SF-C0WG method is formulated  in terms of cell unknowns which are $C^0$ continuous piecewise polynomials of degree $k+2$ with $k\geq 0$ and in terms of face unknowns which are discontinuous piecewise polynomials of degree $k+1$. The formulation of this SF-C0WG method is without the stabilized or penalty term and is as simple as the $C^1$ conforming finite element scheme of the biharmonic problem. Optimal order error estimates in a discrete $H^2$-like norm and the $H^1$ norm for $k\geq 0$ are established for the corresponding WG finite element solutions. Error estimates in the $L^2$ norm are also derived with an optimal order of convergence for $k>0$ and sub-optimal order of convergence for $k=0$. Numerical experiments are shown to confirm the theoretical results.
\end{abstract}


\pagestyle{myheadings}

\watermark{60}{3}{It has been accepted for publication in 
 SCIENCE CHINA Mathematics.}

\section{Introduction}

We consider the biharmonic equation of the form
\begin{subequations}
\begin{eqnarray}
\Delta^2 u&=&f,\quad \mbox{in}\;\Omega,\label{pde}\\
u&=&g_D,\quad\mbox{on}\;\Gamma,\label{pde-bc1}\\
\frac{\partial u}{\partial
\bm{n}}&=&g_N, \quad\mbox{on}\;\Gamma,\label{pde-bc2}
\end{eqnarray}
\end{subequations}
where $\Omega$ is a bounded polytopal domain in $\mathbb{R}^2$ and $\Gamma=\partial\Omega$.

In the case of homogeneous boundary conditions $g_D=g_N=0$, the variational form of problem (\ref{pde})-(\ref{pde-bc2}) reads as: find $u\in H_0^2(\Omega)$ such that
\begin{equation}\label{wf}
(\Delta u, \Delta v) = (f, v),\qquad \forall v\in H_0^2(\Omega),
\end{equation}
where $H_0^2(\Omega)$ is the subspace of $H^2(\Omega)$ consisting of
functions with vanishing value and normal derivative on
$\partial\Omega$.

For the case of nonhomogeneous boundary conditions, assume that $g_D$ and $g_N$ are the Dirichlet boundary data of some function in $H^2(\Omega)$, that is, there exists $\psi\in H^2(\Omega)$ such that
\begin{align*}
\Delta^2 \psi&=0,\quad \mbox{in}\;\Omega,\\
\psi&=g_D,\quad\mbox{on}\;\Gamma,\\
\frac{\partial \psi}{\partial
\bm{n}}&=g_N, \quad\mbox{on}\;\Gamma.
\end{align*}
Then by setting $\widetilde{u}=u-\psi$, we arrive at the weak form (\ref{wf}) for $\widetilde{u}$. Therefore for brevity, but without loss of generality, we will assume homogeneous boundary conditions in the remainder of this paper.

It is well known that $H^2$-conforming finite element methods for problem (\ref{pde})-(\ref{pde-bc2}) involve $C^1$ finite elements, which are of complex implementation and contain high order polynomials even in two dimensions. For example, Argyris and Bell finite elements have 21 and 18 degrees of freedom per triangle, respectively.

In order to avoid the use of such $C^1$ elements, nonconforming finite elements have been used to solve biharmonic problems.
Morley element \cite{morley} is one of the most popular nonconforming finite elements for the biharmonic equations, which only uses quadratic piecewise polynomials on triangle elements in two dimensional domains and doesn't need any stabilization along mesh interfaces. However, it can not be generalized to arbitrarily high order polynomials.

Discontinuous Galerkin (DG) approaches can also be applied to the biharmonic problems.
The first discontinuous Galerkin method---the interior penalty method for the fourth order PDE was presented in \cite{baker}, which uses fully discontinuous piecewise polynomials as basis functions. A nonsymmetric version of interior penalty method was proposed and analyzed in \cite{ms}. Although the DG methods have the advantage of using arbitrarily high-order elements, they also have some disadvantages. The weak forms are more complicated than those used for conforming and nonconforming finite element methods. The discrete linear system of the DG method is large because it has a large number of degrees of freedom.
To reduce the degrees of freedom of DG methods, $C^0$ interior penalty (C0IP) methods have been proposed for the fourth order PDEs first in \cite{engel} and then analyzed in \cite{bs}, where the simple Lagrange elements are used and the continuity of the function derivatives are weakly enforced by stabilization terms on interior edges. However, the C0IP methods still have the disadvantage of complex weak form and the need for the penalty parameters.

Another approach to avoid the use of $C^1$ elements is the mixed methods \cite{ab,falk,monk}, which reduces the biharmonic problem to a system of two second order elliptic problems. One of the main drawbacks of the mixed formulation is that the mixed method leads to saddle-point linear system, {which causes difficulty in efficiently solving the linear algebra system.}

The weak Galerkin (WG) finite element method was first introduced for the second order elliptic problems in \cite{wy}. One of its main characteristics is the use of the concept of weak functions and its weak derivatives. The classical differential operators, such as the gradient and the Laplacian, are approximated by weak differential operator defined as distributions, which are further approximated by piecewise polynomials. These weakly defined functions and differential operators make the WG methods highly flexible in choosing finite element spaces and using polytopal meshes. In recent years, the WG method has been a focus of great interest in the scientific community. Several WG methods have been developed to solve a wide variety of partial differential equations, e.g., \cite{wy-mixed,lyzz,gm,ggz,mwy-helm,mwyz-maxwell,mwyz-interface}. Especially, there are some works \cite{wg-bi1,wg-bi2,wg-bi3,wg-bi4,wg-bi5,wg-bi6,wg-bi7} for biharmonic equations. Compared with the DG methods, there is no penalty parameters needed to tune in the formulation of WG methods. Similar to the DG methods, the WG methods also involve stabilization along mesh skeleton, which makes the implementation of DG and WG methods more complex than the ones of conforming and nonconforming finite element methods.

Most recently, a new WG method without stabilizer term was presented for the second order elliptic problems in \cite{yz}, where we can remove the stabilization and pay the price in the form of using high enough degree of polynomials in the definition of the weak gradient. The resulting numerical scheme is as simple as conforming finite element scheme and it is easy to implement. The idea has been extended to the biharmonic equations in \cite{sf-wg}, where a stabilizer-free WG (SFWG) method has been proposed which uses full discontinuous piecewise polynomials of degrees $k+2$, $k+2$ and $k+1$ with $k\geq 0$, respectively, for discretization of the unknown solution $u$, the trace of $u$ and the trace of normal derivative $\frac{\partial u}{\partial n}$ on the skeleton of the mesh.
For triangular mesh, the minimum degree of polynomials used for the computation of the weak Laplacian is $k+7$ in theory and is $k+4$ in practical computation. As it is pointed out in \cite{sf-wg}, it is a challenging task to compute weak Laplacian and its numerical integration when the degree of polynomials used in the computation of weak Laplacian is very high.

In this paper, we will present and analyze a stabilizer-free $C^0$ weak Galerkin method to approximate the solutions of the biharmonic problem (\ref{pde})- (\ref{pde-bc2}). The method is formulated in terms of face unknowns which are discontinuous piecewise polynomials of degree $k+1$ with $k\geq 0$ and in terms of cell unknowns which are $C^0$ continuous piecewise polynomials of degree $k+2$. We have proved that, for triangular mesh, it is enough to take $k+3$ as the degree of polynomials used in the computation of weak Laplacian.
In comparison with the SFWG method \cite{sf-wg}, the SF-C0WG methods in this paper involve fewer degrees of freedom because nodal values are shared on inter-element boundaries.

The outline of this paper is as follows. In Section \ref{Sec:wg-fem}, we introduce some notations and the formulation of our SF-C0WG method and the related methods. Two energy-like norms and their equivalence and the well-posedness of the SF-C0WG method are discussed in Section \ref{Sec:well-pose}. Then, in Section 4, we derive an error equation which plays an important role in our error estimates. The error analysis of our SF-C0WG method for $H^2$-like norm and the $L^2$ and $H^1$ norm are established in Section \ref{Sec:H2-err} and \ref{Sec:L2-err}, respectively. Finally, in Section \ref{Sec:numeric}, we report some numerical experiment results to confirm the theoretical analysis developed.

\section{Weak Galerkin Finite Element Methods}\label{Sec:wg-fem}

Let ${\mathcal T}_h$ be a quasi-uniform triangulation of the domain $\Omega$.
Denote by ${\cal E}_h$ the set of all edges in ${\cal
T}_h$, and let ${\cal E}_h^0={\cal E}_h\backslash\Gamma$ be
the set of all interior edges.

For convenience, we adopt the following notations,
\begin{eqnarray*}
(v,w)_{\T_h} &=& \sum_{K\in\T_h}(v,w)_K=\sum_{K\in\T_h}\int_K vw d\bx,\\
 \l v,w\r_{\partial\T_h}&=&\sum_{K\in\T_h} \l v,w\r_\pK=\sum_{K\in\T_h} \int_\pK vw ds.
\end{eqnarray*}

{\color{red}For any nonnegative integer $m$}, let $\mathbb{P}_m(D)$ denote the set of polynomials defined on $D$ with degree no more than $m$, where $D$ may be an element $K$ of $\T_h$ or an edge $e$ of $\E_h$. In what follows, we often consider the broken polynomial spaces
\[\mathbb{P}_m(\T_h):=\{v\in L^2(\Omega): v|_K\in \mathbb{P}_m(K), \,\,\forall K\in\T_h\},\]
and
\[\mathbb{P}_m(\E_h):=\{v\in L^2(\E_h): {\color{red}v|_e}\in \mathbb{P}_m(e), \,\,\forall e\in\E_h\}.\]

First of all, we introduce a set of normal directions on ${\cal E}_h$ as follows
\begin{equation}\label{thanks.101}
{\cal D}_h = \{\bm{n}_e: \mbox{ $\bn_e$ is unit and normal to $e$},\
e\in {\cal E}_h \}.
\end{equation}
Then, a weak Galerkin finite element space $V_h$ for $k\geq 0$ is defined by
\begin{equation}
V_h=\{v=\{v_0, v_{n}\bn_e\}:\ v_0\in S_h,
 v_{n}\in \mathbb{P}_{k+1}(\E_h)\},
\end{equation}
with
\begin{align}\label{Sh}
S_h=\{w\in H_0^1(\Omega): w|_K\in \mathbb{P}_{k+2}(K), \,\, \forall K\in\mathcal{T}_h\},
\end{align}
where $v_n$ can be viewed as an approximation of $\frac{\partial v_0}{\partial \bn_e}:=\nabla v_0\cdot\bn_e$.

Denote by $V_h^0$ a subspace of $V_h$ with vanishing traces,
\begin{align}\label{Vh0}
V_h^0=\{v=\{v_0,v_{n}\bn_e\}\in V_h, \ v_{n}|_e=0,\
e\subset\partial K\cap\Gamma\}.
\end{align}

\begin{definition}[Weak Laplacian] For any function $v=\{v_0, v_n\bn_e\}\in V_h$, its weak Laplacian $\Delta_{w,m}v$,
is piecewisely defined as the unique polynomial $(\Delta_{w,m}v)|_K \in \mathbb{P}_{m}(K)$ such that
\begin{equation}\label{wl}
(\Delta_{w,m} v, \ \varphi)_K = -( \nabla v_0, \ \nabla\varphi)_K+\l v_n\bn_e\cdot\bn, \ \varphi\r_{\partial K},\quad
\forall \varphi\in \mathbb{P}_{m}(K),
\end{equation}
for any $K\in\mathcal{T}_h$.
\end{definition}
%In addition, let $\bbQ_h$ be the element-wise defined $L^2$ projection onto $P_{j}(T)$ on each element $T$.

Now, we are ready to present our stabilizer-free $C^0$ weak Galerkin finite element method for the biharmonic problem (\ref{pde})-(\ref{pde-bc2}).

\begin{algorithm}[SF-C0WG Method]
The stabilizer-free $C^0$ weak Galerkin finite element scheme for solving problem (\ref{pde})-(\ref{pde-bc2}) is defined as follows: find $u_h=\{u_0,u_{n}\bn_e\}\in V_h^0$ such that
\begin{equation}\label{wg}
\mathcal{A}_h(u_h, v_h)=(f,\;v_0), \quad\forall\
v_h=\{v_0, v_{n}\bn_e\}\in V_h^0,
\end{equation}
where the bilinear form $a_h(\cdot,\cdot)$ is defined by
\[
\mathcal{A}_h(v, w):=(\Delta_{w,k+3} v,\ \Delta_{w,k+3} w)_{\T_h}, \quad \forall v, w\in V_h.
\]
\end{algorithm}
\begin{remark}
Using the same WG finite element space $V_h^0$ defined by (\ref{Vh0}), a $C^0$ weak Galerkin finite element method has been presented in \cite{wg-bi2}, which is stated as follows:
\begin{algorithm}[C0WG Method]
The $C^0$ weak Galerkin finite element scheme for solving problem (\ref{pde})-(\ref{pde-bc2}) is defined as follows: find $u_h=\{u_0,u_{n}\bn_e\}\in V_h^0$ such that
\begin{equation}\label{c0wg}
\mathcal{A}_{wg}(u_h, v_h)=(f,\;v_0), \quad\forall\
v_h=\{v_0, v_{n}\bn_e\}\in V_h^0,
\end{equation}
where the bilinear form $a_h(\cdot,\cdot)$ is defined by
\[
\mathcal{A}_{wg}(v, w):=(\Delta_{w,k} v,\ \Delta_{w,k} w)_{\T_h}+s_h(v, w), \quad \forall v, w\in V_h,
\]
with the stabilizer term
\[
%s_h(v, w)=\sum_{K\in\T_h}h_K^{-1}\langle \nabla v_0\cdot\bm{n}_e-v_n, \nabla w_0\cdot\bm{n}_e-w_n\rangle_{\partial K}, \quad \forall v, w\in V_h.
s_h(v, w)=\sum_{K\in\T_h}h_K^{-1}\langle \frac{\partial v_0}{\partial\bm{n}_e}-v_n, \frac{\partial w_0}{\partial\bm{n}_e}-w_n\rangle_{\partial K}, \quad \forall v, w\in V_h.
\]
\end{algorithm}

From the formulation of the SF-C0WG method (\ref{wg}) and the C0WG method (\ref{c0wg}), we can see that: the SF-C0WG method is obtained by removing the stabilizer $s_h(\cdot,\cdot )$ in the C0WG method via raising the degree of polynomials used in the definition of the weak Laplacian from $k$ to $k+3$.
A comparison of numerical performance of both WG methods is discussed in Section \ref{Sec:numeric}.
\end{remark}

\begin{remark}
Using the $C^0$ conforming finite element space $S_h$ defined by (\ref{Sh}), a $C^0$ interior penalty  method has been presented in \cite{engel, bs}, which is stated as follows:
\begin{algorithm}[C0IP Method]
The $C^0$ interior penalty method for solving problem (\ref{pde})-(\ref{pde-bc2}) is defined as follows: find $u_h\in S_h$ such that
\begin{equation}\label{c0ipdg}
\mathcal{A}_{dg}(u_h, v_h)=(f,\;v_h), \quad\forall\
v_h\in S_h,
\end{equation}
where the bilinear form $\mathcal{A}_{dg}(\cdot,\cdot)$ is defined as follows: for any $v, w\in S_h$,
\[
\mathcal{A}_{dg}(v, w):=(D^2 v,\ D^2 w)_{\T_h}
-\langle \ljump\nabla v\rjump, \la\frac{\partial^2 w}{\partial \bm{n}_e^2}\ra \rangle_{\E_h}
-\langle \ljump\nabla w\rjump, \la\frac{\partial^2 v}{\partial \bm{n}_e^2}\ra \rangle_{\E_h}
+j_h(v, w),
\]
with the stabilizer term
\[
j_h(v, w)=\sum_{e\in\E_h}\eta h_e^{-1}\langle \ljump\nabla v\rjump, \ljump\nabla w\rjump\rangle_{e}, \quad \forall v, w\in S_h.
\]
Here the penalty parameter $\eta$ is a positive constant.
\end{algorithm}
For any $v\in H^2(\T_h)$, the jump $\ljump \nabla v\rjump$ and the average $\la \frac{\partial^2 v}{\partial \bm{n}_e^2}\ra$  are defined as follows.

Let $e\in \E_h^0$ be the common edge of $K_1$ and $K_2$ of $\T_h$ and $\bm{n}_i, i=1,2$ denote by the outward unit normal vector of the boundary $\partial K_i, i=1,2$. We define on the edge $e$
\begin{align*}
\la \frac{\partial^2 v}{\partial \bm{n}_e^2}\ra=\frac{1}{2}(\frac{\partial^2 v_1}{\partial \bm{n}_e^2}+\frac{\partial^2 v_2}{\partial \bm{n}_e^2})\quad
\mbox{and}\quad
\ljump \nabla v\rjump = \nabla v_1\cdot \bm{n}_1+\nabla v_2\cdot\bm{n}_2,
\end{align*}
where $v_i=v|_{K_i}, i= 1,2$. On a boundary edge $e\subset\partial\Omega$, we simply take $\la \frac{\partial^2 v}{\partial \bm{n}_e^2}\ra=\frac{\partial^2 v}{\partial \bm{n}_e^2}$ and $\ljump \nabla v\rjump = \nabla v\cdot \bm{n}$.

Compared with the C0IP method (\ref{c0ipdg}), our SF-C0WG method (\ref{wg}) has a simple formulation without any integration term on the edges of $\E_h$, which will simplify the implementation of the corresponding numerical scheme and reduce the assembling time of stiffness matrix. Although  the SF-C0WG method (\ref{wg}) has more degrees of freedom than the C0IP method (\ref{c0ipdg}), numerical experiments in Section \ref{Sec:numeric} indicate that its total computational time is less than that of the C0IP method (\ref{c0ipdg}).
\end{remark}

\section{Well Posedness}\label{Sec:well-pose}

For simplicity of notation, from now on we shall drop the subscript $k+3$ in the notation $\Delta_{w,k+3}$ for the discrete weak Laplacian.

In order to analyze the SF-C0WG method (\ref{wg}), we introduce two $H^2$-like norms $\3bar\cdot\3bar$ and $\|\cdot\|_{2,h}$ over $V_h^0$ by
\begin{equation}\label{3barnorm}
\3bar v\3bar=\left[ \sum_{K\in\T_h}\|\Delta_wv\|_{L^2(K)}^2\right]^{1/2},
\end{equation}
and
\begin{equation}\label{norm}
\|v\|_{2,h} = \left[ \sum_{K\in\T_h}\left(\|\Delta v_0\|_{L^2(K)}^2+h_K^{-1}\|\frac{\partial v_0}{\partial \bn_e}-v_n\|^2_{L^2(\pK)}\right) \right]^{1/2},
\end{equation}
for all $v\in V_h^0$. Obviously, $\|\cdot\|_{2,h}$ is indeed a norm on $V_h^0$. We will show $\3bar\cdot\3bar$ is also a norm by proving that the norms $\|\cdot\|_{2,h}$ and $\3bar\cdot\3bar$ are equivalent on the finite element space $V_h^0$ in Lemma \ref{lem:happy}.

In what follows, the trace inequality  is a frequently used analysis tool, which states as \cite{wy-mixed}: for any function $\phi\in H^1(K)$, there holds
\begin{equation}\label{trace}
\|\phi\|_{L^2(\partial K)}^2 \leq C \left( h_K^{-1} \|\phi\|_{L^2(K)}^2 + h_K
\|\nabla \phi\|_{L^2(K)}^2\right).
\end{equation}

The follow lemma plays a key role in the proof of Lemma \ref{lem:happy}.
\begin{lemma} \label{l-m2}
For any $v=\{v_0, v_n\bm{n}_e\}\in V_h$ and $K\in\T_h$, there exists a polynomial $\varphi\in \mathbb{P}_{k+3}(K)$ such that
\begin{align*}
(\Delta v_0, \varphi)_K=0,\quad
\langle (\nabla v_0-v_n\bm{n}_e)\cdot\bm{n}, \varphi\rangle_{\partial K}=\|(\nabla v_0-v_n\bm{n}_e)\cdot\bm{n}\|_{L^2(\partial K)}^2,
\end{align*}
and
\begin{align}\label{q-est}
\|\varphi\|_{L^2(K)}\leq Ch_K^{1/2}\|(\nabla v_0-v_n\bm{n}_e)\cdot\bm{n}\|_{L^2(\partial K)}.
\end{align}
\end{lemma}
\begin{proof}
For any $K\in\T_h$, let $e_i, i=1,2,3$ be the three edges of $K$ and $\lambda_i$'s are the barycentric coordinates of $K$. Then,  we  define a polynomial $\varphi_i\in \mathbb{P}_{k+3}(K)$ for $i=1,2,3$, respectively, by requiring that
\begin{align}\label{q0}
\varphi_i = \prod_{j=1, j\neq i}^3\lambda_j q,
\end{align}
with $q\in \mathbb{P}_{k+1}(K)$ and such that
\begin{subequations}
\begin{align}
\label{qa}
\langle \varphi_i, \tau\rangle_{e_i}&=\langle (\nabla v_0-v_n\bm{n}_e)\cdot\bm{n}, \tau\rangle_{e_i},  \quad\forall \tau\in \mathbb{P}_{k+1}(e_i),\\
\label{qb}
(\varphi_i, \tau)_K &=0,\quad\forall \tau\in \mathbb{P}_k(K).
\end{align}
\end{subequations}
Since there are $$(k+2)+\frac{1}{2}(k+1)(k+2)=\frac{1}{2}(k+2)(k+3)$$ equations and the same number of unknowns in the linear system (\ref{qa})-(\ref{qb}), the existence and uniqueness of $\varphi_i$ are equivalent.

Assume that both $\varphi_i$ and $\widehat{\varphi}_i$ satisfy the linear system (\ref{qa})-(\ref{qb}), we will prove their difference $d_i=\varphi_i-\widehat{\varphi}_i$ vanishes on $K$. From (\ref{q0})-(\ref{qb}), we know that $d_i$ can be expressed as
\begin{align}\label{d0}
d_i = \prod_{j=1, j\neq i}^3\lambda_j \widetilde{q},
\end{align}
with $\widetilde{q}\in \mathbb{P}_{k+1}(K)$ and
satisfies the following conditions,
\begin{subequations}
\begin{align}
\label{da}
\langle d_i, \tau\rangle_{e_i}&=0,  \quad\forall \tau\in \mathbb{P}_{k+1}(e_i),\\
\label{db}
(d_i, \tau)_K &=0,\quad\forall \tau\in \mathbb{P}_k(K).
\end{align}
\end{subequations}
It follows from (\ref{da}) that $d_i=0$ on $e_i$, which together with (\ref{d0}) implies $\widetilde{q}$ in (\ref{d0}) can be written as $\widetilde{q}=\lambda_i \omega$ with $\omega\in \mathbb{P}_{k}(K)$. Therefore, we have
\begin{align*}
d_i = \prod_{j=1}^3\lambda_j \omega, \,\, \mbox{ with } \omega\in \mathbb{P}_{k}(K),
\end{align*}
which combining with (\ref{db}) implies $d_i=0$ on $K$.

Hence, the linear system (\ref{qa})-(\ref{qb}) has a unique solution  $\varphi_i$ in the form of (\ref{q0}), which belongs to $\mathbb{P}_{k+3}(K)$.

Then, by a scaling arguments, we have
\begin{align*}
\|\varphi_i\|_{L^2(K)} \leq Ch_K^{1/2}\|\varphi_i\|_{L^2(\partial K)}.
\end{align*}
Thanks to (\ref{q0}), it is known that $\varphi_i = 0$ on $e_j$ for $j\neq i$. Then,
$\|\varphi_i\|_{L^2(\partial K)}=\|\varphi_i\|_{L^2(e_i)}$. Therefore,
\begin{align}\label{s1}
\|\varphi_i\|_{L^2(K)} \leq Ch_K^{1/2}\|\varphi_i\|_{L^2(e_i)}.
\end{align}
Let $\theta_i(x) = \prod_{j=1,j\neq i}^{3}\lambda_j(x)$. Using (\ref{q0}) and (\ref{qa}), we have
\begin{align*}
\langle  \theta_i q, \tau\rangle_{e_i}
=\langle \varphi_i, \tau\rangle_{e_i}
\leq \|(\nabla v_0-v_n\bm{n}_e)\cdot\bm{n}\|_{L^2(e_i)}\|\tau\|_{L^2(e_i)}.
\end{align*}
Taking $\tau = q$ in the above inequality, and by the second mean value theorem of integrals, there exist a point $\varepsilon_1\in e_i$ such that
\begin{align*}
\theta_i(\varepsilon_1)\|q\|_{L^2(e_i)}^2=\langle  \theta_i q, q\rangle_{e_i}
\leq \|(\nabla v_0-v_n\bm{n}_e)\cdot\bm{n}\|_{L^2(e_i)}\|q\|_{L^2(e_i)}.
\end{align*}
Then, after cancelling $\|q\|_{L^2(e_i)}$, we obtain
\begin{align}\label{s2}
\|q\|_{L^2(e_i)}
\leq \theta_i^{-1}(\varepsilon_1)\|(\nabla v_0-v_n\bm{n}_e)\cdot\bm{n}\|_{L^2(e_i)}.
\end{align}
Therefore, using (\ref{q0}) and the second mean value theorem of integrals again, there exist a point $\varepsilon_2\in e_i$ such that
\begin{align*}
\|\varphi_i\|_{L^2(e_i)}=\sqrt{\langle \theta_i^2, q^2\rangle_{e_i}}
=\theta_i(\varepsilon_2)\|q\|_{L^2(e_i)},
\end{align*}
which together with (\ref{s2}) leads to
\begin{align*}
\|\varphi_i\|_{L^2(e_i)} \leq \theta_i(\varepsilon_2)\theta_i^{-1}(\varepsilon_1)
\|(\nabla v_0-v_n\bm{n}_e)\cdot\bm{n}\|_{L^2(e_i)}.
\end{align*}
Thus, from (\ref{s1}) and the above inequality, we obtain
\begin{align*}
\|\varphi_i\|_{L^2(K)} \leq Ch_K^{1/2}\|(\nabla v_0-v_n\bm{n}_e)\cdot\bm{n}\|_{L^2(e_i)}.
\end{align*}
Finally, choosing $\varphi = \sum_{i=1}^3\varphi_i$ ends the proof.
\end{proof}

\begin{lemma}\label{lem:happy} There exist two positive constants $C_1$ and $C_2$ such
that for any $v=\{v_0,v_n\bn_e\}\in V_h$, we have
\begin{equation*}
C_1 \|v\|_{2,h}\le \3bar v\3bar \leq C_2 \|v\|_{2,h}.
\end{equation*}
\end{lemma}

\begin{proof}
For any $v=\{v_0,v_n\bn_e\}\in V_h$ and $\varphi\in \mathbb{P}_{k+3}(K)$, it follows from the definition of
weak Laplacian (\ref{wl}) and integration by parts that
\begin{eqnarray}
(\Delta_{w} v, \varphi)_K
&=&-(\nabla v_0, \nabla\varphi)_K+\l v_n\bn_e\cdot\bn,  \varphi\r_\pK\nonumber\\
&=&(\Delta v_0, \varphi)_K+\l (v_n\bn_e-\nabla v_0)\cdot\bn, \varphi\r_\pK.\label{n-1}
\end{eqnarray}

By letting $\varphi=\Delta_w v$ in (\ref{n-1}) we arrive at
\begin{eqnarray*}
\|\Delta_{w} v\|^2_{L^2(K)} &=&(\Delta v_0,  \Delta_w v)_K+\l (v_n\bn_e-\nabla v_0)\cdot\bn,  \Delta_w v\r_\pK.
\end{eqnarray*}

From the trace inequality (\ref{trace}) and the inverse inequality,
we have
\begin{eqnarray*}
\|\Delta_wv\|^2_{L^2(K)} &\le& \|\Delta v_0\|_{L^2(K)} \|\Delta_w v\|_{L^2(K)}
   +\|(v_n\bn_e-\nabla v_0)\cdot\bn\|_{L^2(\pK)} \|\Delta_w v\|_{L^2(\pK)}\\
&\le& C(\|\Delta v_0\|_{L^2(K)}+h_K^{-1/2}\|(v_n\bn_e-\nabla v_0)\cdot\bn\|_{L^2(\pK)}) \|\Delta_w v\|_{L^2(K)},
\end{eqnarray*}
which implies
$$
\|\Delta_w v\|_{L^2(K)} \le C \left(\|\Delta v_0\|_{L^2(K)}+h_K^{-1/2}\|(v_n\bn_e-\nabla v_0)\cdot\bn\|_{L^2(\pK)}\right),
$$
and consequently
$$\3bar v\3bar \leq C_2 \|v\|_{2,h}.$$
Next we will prove
\begin{equation}\label{n-33}
\sum_{K\in\T_h} h_K^{-1}\|(\nabla v_0-v_n\bn_e)\cdot\bn\|^2_{L^2(\pK)} \leq C\3bar v\3bar^2.
\end{equation}

%By Lemma \ref{l-m1}, there exist a $\varphi_0$ such that for each $e\subset\pK$,
%\begin{eqnarray}\label{cond1}
%(\Delta v_0,\varphi_0)_K=0,\;\;\l (v_n\bn_e-\nabla v_0)\cdot\bn, \ \varphi_0\r_\pK=\|(v_n\bn_e-\nabla v_0)\cdot\bn\|_e^2,
%\end{eqnarray}
%and
%\begin{eqnarray}\label{cond2}
%\|\varphi_0\|_K\le C h_K^{1/2}\|(v_n\bn_e-\nabla v_0)\cdot\bn\|_e.
%\end{eqnarray}

Let $\varphi_0$ be obtained from Lemma \ref{l-m2}, taking $\varphi=\varphi_0$ in (\ref{n-1}) yields
\begin{align}
\|(v_n\bn_e-\nabla v_0)\cdot\bn\|_{L^2(\partial K)}^2
&=(\Delta_{w} v, \varphi_0)_K\le \|\Delta_{w} v\|_{L^2(K)}\|\varphi_0\|_{L^2(K)}\nonumber\\
&\leq Ch_K^{1/2} \|\Delta_{w} v\|_{L^2(K)}\|(v_n\bn_e-\nabla v_0)\cdot\bn\|_{L^2(\partial K)},
\end{align}
which implies (\ref{n-33}).

Finally, by letting $\varphi=\Delta v_0$ in (\ref{n-1}) we arrive at
\begin{eqnarray*}
\|\Delta v_0\|^2_{L^2(K)} &=&(\Delta v_0, \ \Delta_w v)_K
    -  \l (v_n\bn_e-\nabla v_0)\cdot\bn, \ \Delta_w v\r_\pK
\end{eqnarray*}
Using the trace inequality (\ref{trace}), the inverse inequality, and (\ref{n-33}), one has
\begin{eqnarray*}
\|\Delta v_0\|^2_{L^2(K)} &\le& C \|\Delta_w v\|_{L^2(K)}\|\Delta v_0\|_{L^2(K)},
\end{eqnarray*}
which gives
\begin{eqnarray}
\sum_{K\in\T_h}\|\Delta v_0\|^2_{L^2(K)} &\le& C \3bar v\3bar^2,
\end{eqnarray}
which together with (\ref{n-33}) yields
\[ \3bar v\3bar \geq C_1\|v\|_{2,h}.\]
The proof is completed.
\end{proof}

\smallskip
In the following lemma we will prove the well-posedness of the SF-C0WG method (\ref{wg}).
\begin{lemma}
The SF-C0WG finite element scheme (\ref{wg}) has a unique
solution.
\end{lemma}
\begin{proof}
To show the well-posedness of (\ref{wg}) assume that $f=g_D=g_N=0$. We will show that $u_h$ vanishes. Take $v=u_h$  in (\ref{wg}). It follows that
\[
(\Delta_w u_h,\Delta_w u_h)_{\T_h}=0.
\]
Then Lemma \ref{lem:happy} implies $\|u_h\|_{2,h}=0$. Consequently, we have $\Delta u_0=0$,
$\nabla u_0\cdot\bn_e=u_{n}$ on $\pK$.
Thus $u_0$ is the solution of (\ref{pde})-(\ref{pde-bc2}) with $f=g_D=g_N=0$. We have $u_0=0$, then $u_n=0$, which ends the proof.
\end{proof}

\section{An Error Equation}\label{Sec:err-eqn}

Let $Q_0:H^2(\Omega)\rightarrow S_h$ be the Scott-Zhang interpolation operator introduced in \cite{wg-bi2},
which has the following properties:
\begin{itemize}
\item [a.] \citep[Page 493]{wg-bi2} $Q_0$ preserves polynomial of degree up to $k+2$, i.e., $Q_0v=v\in \mathbb{P}_{k+2}(\T_h)$.\\
\item [b.] \citep[Lemma 8.2]{wg-bi2} $Q_0$ preserves the face mass of order $k$, i.e.,
\begin{align}\label{p-fm}
\langle v-Q_0v, p\rangle_{e}=0,\quad \forall p\in \mathbb{P}_k(e), \, e\in \E_h.
\end{align}
\item [c.] \citep[Theorem 8.1]{wg-bi2}For any $v\in H^{\gamma}(\Omega)$ with $\gamma\geq 2$, there holds
\begin{align}\label{Q0-approxi}
\left[\sum_{K\in\T_h}h^{2s}|v-Q_0v|_{H^s(K)}^2\right]^{1/2}\leq Ch^{\min\{k+3,\gamma\}}|v|_{H^{\gamma}(\Omega)},\quad 0\leq s\leq 2.
\end{align}
\end{itemize}

Now for the true solution $u$ of (\ref{pde})-(\ref{pde-bc2}), we introduce an interpolation operator $Q_h: H^2(\Omega)\rightarrow V_h$ such that on each element $K\in\mathcal{T}_h$,
\[
%Q_h u = \{Q_0 u, Q_n(\nabla u\cdot\bn_e)\bn_e\},
Q_h u = \{Q_0 u, Q_n(\frac{\partial u}{\partial \bn_e})\bn_e\},
\]
where $Q_n$ denotes the element-wise defined $L^2$ projections from $L^2(e)$ onto $\mathbb{P}_{k+1}(e)$ for each $e\subset \pK$.

Define the error between the WG solution $u_h=\{u_0, u_n\bm{n}_e\}$ and the projection $Q_hu=\{Q_0u, Q_n(\frac{\partial u}{\partial\bm{n}_e})\bm{n}_e\}$ of the exact solution $u$ as
\[
 e_h=Q_hu-u_h:=\{e_0, e_n\bm{n}_e\},
\]
with
\[
e_0=Q_0u-u_0, \quad e_n=Q_n(\frac{\partial u}{\partial\bm{n}_e})-u_n.
\]
The aim of this section is to obtain an error equation that $e_h$ {\color{red}satisfied}.

\begin{lemma}
Let $\pi_h$ be an element-wise defined $L^2$ projections onto $\mathbb{P}_{k+3}(K)$ on each element $K\in\T_h$.
For any $K\in\T_h$ and $w\in H^2(\Omega)$, we have
\begin{align}\label{key}
(\Delta_{w}(Q_hw), v)_K = (\Delta Q_0w,\ v)_K
 +\langle  Q_n(\frac{\partial w}{\partial \bm{n}})-\frac{\partial }{\partial\bn}(Q_0w), v \rangle_{\pK},
\end{align}
for any $v\in \mathbb{P}_{k+3}(K)$.
\end{lemma}
\begin{proof}
From the definition (\ref{wl}) of weak Laplacian it follows that
\begin{align}\label{q1}
(\Delta_{w}(Q_hw), v)_K &= -(\nabla Q_0w,\ \nabla v)_K +\langle  Q_n(\frac{\partial w}{\partial \bm{n}_e})\bm{n}_e\cdot\bm{n}, v \rangle_{\pK},
\end{align}
for any $v\in \mathbb{P}_{k+3}(K)$.

Using integration by parts, we get
\begin{align}\label{q2}
-(\nabla Q_0w,\ \nabla v)_K = (\Delta Q_0w, \ v)_K -\langle \nabla Q_0w\cdot\bm{n}, w\rangle_{\partial K}.
\end{align}
Plugging (\ref{q2}) into (\ref{q1}), and recalling that
\[Q_n(\frac{\partial w}{\partial \bm{n}_e})\bm{n}_e\cdot\bm{n}=Q_n(\frac{\partial w}{\partial \bm{n}})\]
yields (\ref{key}).
The proof is completed.
\end{proof}

\begin{lemma}[Error Equation]
Let $u$ and $u_h$ be the solutions of the problem (\ref{pde})-(\ref{pde-bc2}) and the SF-C0WG scheme (\ref{wg}), respectively.
For any $v\in V_h^0$, we have
\begin{eqnarray}
\mathcal{A}_h(e_h, v)=\ell(u,v),\label{ee}
\end{eqnarray}
where $\ell(u,v):=\sum\nolimits_{i=1}^2\ell_i(u,v)$, with
\begin{subequations}
  \begin{align}
  \label{el1}
  \ell_1(u,v)&:=(\Delta_w (Q_hu)-\pi_h\Delta u, \Delta_w v)_{\T_h},\\
  \label{el2}
  \ell_2(u,v)&:= \langle \Delta u-\pi_h\Delta  u, (\nabla v_0-v_n\bn_e)\cdot\bn\rangle_{\partial\T_h}.
  \end{align}
\end{subequations}
\end{lemma}

\begin{proof}
For $v=\{v_0,v_n\bn_e\}\in V_h^0$, testing (\ref{pde}) by  $v_0$  and using the fact that
\[\sum_{K\in\T_h}\langle \Delta u, v_n\bn_e\cdot\bn\rangle_\pK=0\]
and integration by parts,  we arrive at
\begin{eqnarray}
(f,v_0)&=&(\Delta^2u, v_0)_{\T_h}\nonumber\\
&=&(\Delta u,\Delta v_0)_{\T_h} -\langle \Delta u,
\nabla v_0\cdot\bn\rangle_{\partial\T_h} + \langle\nabla(\Delta u)\cdot\bn,
v_0\rangle_{\partial\T_h}\label{m1}\\
&=&(\Delta u,\Delta v_0)_{\T_h} -\langle \Delta u,
(\nabla v_0-v_n\bn_e)\cdot\bn\rangle_{\partial\T_h}.\nonumber
\end{eqnarray}
Next we investigate the term  $(\Delta u,\Delta v_0)_{\T_h}$ in the above equation. Using (\ref{key}), integration by parts and the definition of weak Laplacian, we have
\begin{eqnarray*}
 (\Delta u, \Delta v_0)_{\T_h}&=&(\pi_h\Delta u, \Delta v_0)_{\T_h} \\
&=& -(\nabla v_0, \nabla (\pi_h\Delta u))_{\T_h}
    +\langle \nabla v_0\cdot\bn, \pi_h\Delta u \rangle_{\partial\T_h}\nonumber\\
&=&(\Delta_w v,\ \pi_h\Delta u)_{\T_h}
  +\langle (\nabla v_0-v_n\bn_e)\cdot\bn, \pi_h\Delta
u\rangle_{\partial\T_h}\nonumber\\
&=&(\Delta_w (Q_hu),\ \Delta_w v)_{\T_h}-\ell_1(u,v)
   +\langle (\nabla v_0-v_n\bn_e)\cdot\bn, \pi_h\Delta
u\rangle_{\partial\T_h},
\end{eqnarray*}
which together with (\ref{m1}) yields
\begin{eqnarray}
(f,v_0)&=&\mathcal{A}_h(Q_hu, v)-\ell_1(u,v)
-\langle (\nabla v_0-v_n\bn_e)\cdot\bn, \Delta u-\pi_h\Delta u \rangle_{\partial\T_h}.\label{mmmm}
\end{eqnarray}
which implies that
\begin{eqnarray*}
\mathcal{A}_h(Q_hu, v)=(f,v_0)+\sum_{i=1}^2\ell(u,v).
\end{eqnarray*}
Subtracting (\ref{wg}) from the above equation ends the proof.
\end{proof}

\section{An Error Estimate in the $H^2$-like Norm}\label{Sec:H2-err}

We will obtain the optimal convergence rate for the solution $u_h$ of the SF-C0WG method (\ref{wg}) in a discrete $H^2$ norm.

\begin{lemma}\label{l2}
Assume $w\in H^{\gamma+2}(\Omega)$ with $\gamma>0$. There exists a constant $C$ such that the following estimates hold true:
\begin{align}
\label{mmm1}
&\left(\sum_{K\in\T_h} h_K\|\Delta w-\pi_h\Delta w\|_{L^2(\partial
K)}^2\right)^{1/2}
\leq C h^{\min\{k+4,\gamma\}}|w|_{H^{\gamma+2}(\Omega)},\\
\label{z2}
&\left(\sum_{K\in\T_h} h^{-1}_K\|\frac{\partial }{\partial\bn}(Q_0w)-Q_n(\frac{\partial w}{\partial\bn})\|_{L^2(\pK)}^2\right)^{1/2}
\leq Ch^{\min\{k+1,\gamma\}}|w|_{H^{\gamma+2}(\Omega)},\\
\label{z3}
&\|\Delta_w(Q_hw)-\pi_h\Delta w\|_{L^2(\T_h)}\leq Ch^{\min\{k+1,\gamma\}}|w|_{H^{\gamma+2}(\Omega)}.
\end{align}
\end{lemma}
\begin{proof}
By the trace inequality (\ref{trace}) and the approximation property of the $L^2$ orthogonal projection $\pi_h$, we have
\begin{align*}
&h_K\|\Delta w - \pi_h \Delta w\|_{L^2(\partial K)}^2\\
&\leq C(\|\Delta w - \pi_h \Delta w\|_{L^2(K)}^2+
h_K^2\|\nabla(\Delta w - \pi_h \Delta w)\|_{L^2(K)}^2)\\
&\leq Ch_K^{2\min\{k+4,\gamma\}}|\Delta w|_{H^{\gamma}(K)}^2\\
&\leq Ch_K^{2\min\{k+4,\gamma\}}|w|_{H^{\gamma+2}(K)}^2.
\end{align*}
Taking the summation of the above inequalities over all $K\in\T_h$,  we completes the proof of (\ref{mmm1}).

Next, we turn to the estimate (\ref{z2}).
It follows from the definition of $Q_0$ and $Q_n$ that
\begin{align}\label{z1}
&\|\frac{\partial}{\partial\bn}(Q_0w)-Q_n(\frac{\partial w}{\partial\bm{n}})\|_{L^2(\pK)}\nonumber\\
%&=\|\frac{\partial}{\partial \bm{n}}(Q_0w)-Q_n(\frac{\partial w}{\partial\bm{n}})\|_{L^2(\pK)}\nonumber\\
&\leq \|\frac{\partial}{\partial\bm{n}} (Q_0w-w)\|_{L^2(\pK)}
+\|\frac{\partial w}{\partial\bm{n}}-Q_n(\frac{\partial w}{\partial\bm{n}})\|_{L^2(\pK)}\nonumber\\
&\leq 2\|\frac{\partial}{\partial\bm{n}} (Q_0w-w)\|_{L^2(\pK)}.
\end{align}
Furthermore, using the trace inequality (\ref{trace}) and the approximation property (\ref{Q0-approxi}) of $Q_0$, we obtain
\begin{align*}
&\|\frac{\partial}{\partial\bm{n}} (Q_0w-w)\|_{L^2(\pK)}^2\\
&\leq  C(h_K^{-1}\|\nabla (Q_0w-w)\|_{L^2(K)}^2+h_K\|\nabla^2 (Q_0w-w)\|_{L^2(K)}^2)\\
&\leq Ch_K^{\min\{2k+3, 2\gamma+1\}}|w|_{H^{\gamma+2}(K)}^2,
\end{align*}
which together with (\ref{z1}) yields
\begin{align*}
\sum_{K\in\T_h} h^{-1}_K\|\frac{\partial }{\partial\bn}(Q_0w)-Q_n(\frac{\partial w}{\partial\bn})\|_{L^2(\pK)}^2
\leq Ch^{\min\{2k+2,2\gamma\}}|w|_{H^{\gamma+2}(\Omega)}^2,
\end{align*}
which ends the proof of (\ref{z2}).

Now we consider the estimate (\ref{z3}). For any $v\in \mathbb{P}_{k+3}(\T_h)$,
from (\ref{key}) and the orthogonal property of the $L^2$ projection $\pi_h$, it follows  that
\begin{align}\label{c1}
&(\Delta_w(Q_hw)-\pi_h w, v)_{\T_h}\nonumber\\
&=(\Delta (Q_0w-w), v)_{\T_h}
+\langle Q_n(\frac{\partial u}{\partial\bm{n}})-\frac{\partial}{\partial\bn} (Q_0u), v \rangle_{\partial \T_h}\nonumber\\
&=I_1+I_2.
\end{align}
From the Cauchy-Schwarz inequality and the approximation property (\ref{Q0-approxi}) of $Q_0$, one has
\begin{align}\label{c2}
|I_1|
&\leq \sum_{K\in\T_h} \|\Delta (Q_0w-w)\|_{L^2(K)}\| v\|_{L^2(K)}\nonumber\\
&\leq (\sum_{K\in\T_h} |Q_0w-w|_{H^2(K)}^2)^{1/2}\| v\|_{L^2(\T_h)}\nonumber\\
&\leq Ch^{\min\{k+1,\gamma\}}|w|_{H^{\gamma+2}(\Omega)}\| v\|_{L^2(\T_h)}.
\end{align}
Using the Cauchy-Schwarz inequality, (\ref{z2}) and the inverse inequality, we arrive at
\begin{align*}
|I_2|
&\leq (\sum_{K\in\T_h}h_K^{-1}\|(Q_n(\frac{\partial u}{\partial\bm{n}})-\frac{\partial}{\partial \bn} (Q_0u)\|_{L^2(\partial K)}^2)^{1/2}
 (\sum_{K\in\T_h}h_K\| v\|_{L^2(\partial K)}^2)^{1/2}\\
&\leq Ch^{\min\{k+1,\gamma\}}|w|_{H^{\gamma+2}(\Omega)}\| v\|_{L^2(\T_h)},
\end{align*}
which together with (\ref{c1}) and (\ref{c2}) yields
\begin{align*}
|(\Delta_w(Q_hw)-\pi_h w, v)_{\T_h}|\leq Ch^{\min\{k+1,\gamma\}}|w|_{H^{\gamma+2}(\Omega)}\| v\|_{L^2(\T_h)}.
\end{align*}
Taking $v=\Delta_w(Q_hw)-\pi_h w$ in the above inequality ends the proof of (\ref{z3}).
\end{proof}

\begin{lemma}\label{l3}
Assume $w\in H^{\gamma+2}(\Omega)$ with $\gamma>0$. There exists a constant $C$ such that the following estimates hold true:
\begin{align}
\label{mm1}
|\ell_1(w, v)|&\leq Ch^{\min\{k+1,\gamma\}}|w|_{H^{\gamma+2}(\Omega)}\3bar v\3bar,\\
\label{mm2}
|\ell_2(w, v)|&\leq Ch^{\min\{k+4,\gamma\}}|w|_{H^{\gamma+2}(\Omega)}\3bar v\3bar,
\end{align}
for any $v\in V_h^0$.
\end{lemma}
\begin{proof}
Using the Cauchy-Schwarz inequality and (\ref{z3}) of Lemma \ref{l2}, we have
\begin{align*}
|\ell_1(w, v)|&=\left| (\Delta_w(Q_hw)-\pi_h w, \Delta_wv)_{\T_h}\right|\\
&\leq \|\Delta_w(Q_hw)-\pi_h w\|_{L^2(\T_h)} \|\Delta_wv\|_{L^2(\T_h)}\\
&\leq Ch^{\min\{k+1,\gamma\}}|w|_{H^{\gamma+2}(\Omega)}\3bar v\3bar.
\end{align*}

It follows from the Cauchy-Schwarz inequality, (\ref{mmm1}), and Lemma \ref{lem:happy} that
\begin{align*}
|\ell_2(w,v)|&=\left|\sum_{K\in\T_h} \langle \Delta w-\pi_h\Delta w, (\nabla
v_0-v_n\bn_e)\cdot\bn\rangle_\pK\right|\nonumber\\
&\leq \left(\sum_{K\in\T_h} h_K\|\Delta w-\pi_h\Delta
w\|_{L^2(\pK)}^2\right)^{1/2} \nonumber\\
&\quad\times\left(\sum_{K\in\T_h} h_K^{-1}
\|(\nabla v_0-v_n\bn_e)\cdot\bn\|_{L^2(\pK)}^2\right)^{1/2}\nonumber\\
&\leq C h^{\min\{k+4,\gamma\}}|w|_{H^{\gamma+2}(\Omega)} \|v\|_{2,h}\nonumber\\
&\leq C h^{\min\{k+4,\gamma\}}|w|_{H^{\gamma+2}(\Omega)} \3bar v\3bar.
\end{align*}
We have completed the proof.
\end{proof}

%\begin{lemma}
%Let  $u\in H^{k+3}(\Omega)$,  then
%\begin{equation}\label{eee2}
%\| u-Q_hu\|_{2,h}\leq Ch^{k+1}\|u\|_{k+3}.
%\end{equation}
%\end{lemma}
%\begin{proof}
%For the sake of simplicity, we denote $\eta:=u-Q_hu$.
%For any $K\in\T_h$, it follows from the definition of the norm $\|\cdot\|_{2,h}$ and (\ref{Q0-approxi}) that
%\begin{align}
%\|\Delta(u-Q_0u)\|_K\leq \|u-Q_0u\|_{2,K}\leq Ch_K^{k+1}\|u\|_{k+3,K}
%\end{align}

%By the trace inequality (\ref{trace}), we have
%\begin{align*}
%\|(\nabla (u-Q_0u)-Q_n(\nabla (u-Q_0u)\cdot\bm{n}_e)\bm{n}_e)\bm{n}\|_{\partial K}
%\end{align*}
%Using the above inequality and taking the summation of it over $T$,  we derive (\ref{eee2}) and prove the lemma.
%\end{proof}

\begin{theorem}\label{thm1}
Let $u_h\in V_h$  be the solution arising from the SF-C0WG scheme
(\ref{wg}). Assume that the exact solution $u\in H^{k+3}(\Omega)$. Then, there
exists a constant $C$ such that
\begin{equation}\label{err1}
\3bar Q_hu-u_h\3bar \leq Ch^{k+1}|u|_{H^{k+3}(\Omega)}.
\end{equation}
\end{theorem}
\begin{proof}
Taking $v=e_h$ in the error equation (\ref{ee}) and using Lemma \ref{l3} with $\gamma=k+1$, we arrive at
\begin{align*}
\3bar e_h\3bar^2&=\ell(u, e_h)\leq Ch^{k+1}|u|_{H^{k+3}(\Omega)}\3bar e_h\3bar,
\end{align*}
which completes the proof.
\end{proof}

\section{Error Estimates in the $L^2$ Norm and $H^1$ Norm}\label{Sec:L2-err}

In this section, we will provide estimates for the
$L^2$ norm and $H^1$ norm of the error between the exact solution $u$ and its corresponding WG finite element solution $u_h$.

Firstly, let us introduce the following dual problem
\begin{eqnarray}
\Delta^2\phi&=& \chi\quad
\mbox{in}\;\Omega,\label{dual}\\
\phi&=&0\quad\mbox{on}\;\Gamma,\label{dual1}\\
\nabla \phi\cdot\bn&=&0\quad\mbox{on}\;\Gamma.\label{dual2}
\end{eqnarray}
Assume that the dual problem has the $H^{\alpha+2}$-regularity in the sense that there exists a constant $C$ such that
\begin{equation}\label{reg}
\|\phi\|_{H^{\alpha+2}(\Omega)}\le C\|\chi\|_{H^{\alpha-2}(\Omega)},\quad \mbox{ for } \alpha=1,2.
\end{equation}
For $\chi\in H^{\alpha-2}(\Omega)$ with $\alpha>0$, the $H^{\alpha+2}$-regularity has been proved for smooth domains in any dimension\cite{dauge}. The $H^4$-regularity has been proved by Blum and Rannacher in \cite{br}  for the two dimensional convex polygonal domains with inner angles less than $126.28\dots^{\rm o}$.

%%lem:p1
%\begin{lemma}\label{lem:p1}
%For any $K\in\T_h$, there holds
%\begin{align}\label{pp}
%\Delta_w(Q_hw) = \Delta w,\quad \forall w\in \mathbb{P}_{k+2}(K).
%\end{align}
%\end{lemma}

%%lem:app2
%Let $\mathcal{P}_h^{k+2}: L^2(\T_h)\rightarrow \mathbb{P}_{k+2}(\T_h)$ be the piecewisely defined $L^2$ orthogonal projection.
%\begin{lemma}\label{lem:p2}
%Assume $\phi\in H^{\alpha+2}(\Omega)$ with $\alpha=1,2$. Then, there holds
%\begin{align}\label{p2}
%\|\Delta_wQ_h(\phi-\mathcal{P}_h^{k+2}\phi)\|_{L^2(\T_h)}
%\leq Ch^{\min\{k+1,\alpha\}}|\phi|_{H^{\alpha+2}(\Omega)}.
%\end{align}
%\end{lemma}

%lem:app3
\begin{lemma}\label{lem:p3}
Let $\phi\in H^{\alpha+2}(\Omega)$ with $\alpha=1,2$. Then, there holds
\begin{align}\label{p3}
|\Delta_w(Q_h\phi)|_{H^{\alpha}(\T_h)}\leq Ch^{\min\{k+1-\alpha,0\}}|\phi|_{H^{\alpha+2}(\Omega)}.
\end{align}
\end{lemma}
\begin{proof}
The proof is given in Appendix.
\end{proof}

\begin{lemma}\label{l4}
Assume $u\in H^{k+3}(\Omega)$ and $\phi\in H^{\alpha+2}(\Omega)$ with $\alpha=1,2$. Then for $k\geq 0$, there holds
\begin{align}
\label{a1}
|\ell_1(u, Q_h\phi)|&\leq Ch^{\min\{2k+2,k+1+\alpha\}}|u|_{H^{k+3}(\Omega)}|\phi|_{H^{\alpha+2}(\Omega)},\\
\label{a2}
|\ell_2(u, Q_h\phi)|&\leq Ch^{\min\{2k+2,k+1+\alpha\}}|u|_{H^{k+3}(\Omega)}|\phi|_{H^{\alpha+2}(\Omega)}.
\end{align}
\end{lemma}
\begin{proof}
Let $\mathcal{P}_h^{\alpha-1}$ be the $L^2$ orthogonal projection onto the piecewise polynomial space $\mathbb{P}_{\alpha-1}(\T_h)$. For simplicity, denote by $\phi_h=\Delta_w(Q_h\phi)$ and $\widehat{\phi}_h=\mathcal{P}_h^{\alpha-1}(\phi_h)$. Then,
\begin{align}\label{l1-est}
\ell_1(u, Q_h\phi)
&= (\Delta_w(Q_hu)-\pi_h\Delta u, \phi_h )_{\T_h}\nonumber\\
&= (\Delta_w(Q_hu)-\pi_h\Delta u, \phi_h-\widehat{\phi}_h )_{\T_h}
+(\Delta_w(Q_hu)-\pi_h\Delta u, \widehat{\phi}_h )_{\T_h}\nonumber\\
&=T_1+T_2.
\end{align}
Using the Cauchy-Schwarz inequality, (\ref{z3}) of Lemma \ref{l2} and (\ref{p3}), one has
\begin{align}\label{T1-est}
|T_1|&=|(\Delta_w(Q_hu)-\pi_h\Delta u, \phi_h-\widehat{\phi}_h )_{\T_h}|\nonumber\\
&\leq \|\Delta_w(Q_hu)-\pi_h\Delta u\|_{L^2(\T_h)}\|\phi_h-\widehat{\phi}_h\|_{L^2(\T_h)}\nonumber\\
&\leq Ch^{k+1}|u|_{H^{k+3}(\Omega)}\cdot h^{\alpha} |\phi_h|_{H^{\alpha}(\Omega)}\nonumber\\
&\leq Ch^{k+1+\alpha}|u|_{H^{k+3}(\Omega)}\cdot h^{\min\{k+1-\alpha,0\}}|\phi|_{H^{\alpha+2}(\Omega)}\nonumber\\
&\leq Ch^{\min\{2k+2,k+1+\alpha\}}|u|_{H^{k+3}(\Omega)} |\phi|_{H^{\alpha+2}(\Omega)}.
\end{align}
Now we turn to the estimate of the term $T_2$. Firstly, we rewrite $T_2$ as follows:
\begin{align}\label{T2}
T_2&=(\Delta_w(Q_hu)-\pi_h\Delta u, \widehat{\phi}_h )_{\T_h}\nonumber\\
&=(\Delta (Q_0u-u), \widehat{\phi}_h)_{\T_h}
+\langle Q_n(\frac{\partial u}{\partial \bn})-\frac{\partial}{\partial\bn}(Q_0u), \widehat{\phi}_h\rangle_{\partial\T_h}\nonumber\\
&=-(\nabla (Q_0u-u), \nabla\widehat{\phi}_h)_{\T_h}
+\langle Q_n(\frac{\partial u}{\partial \bn})-\frac{\partial u}{\partial\bn}, \widehat{\phi}_h\rangle_{\partial\T_h}\nonumber\\
&=J_1+J_2.
\end{align}
For the first term $J_1$, we discuss it in the following two cases:
\begin{itemize}
\item In the case of $\alpha=1$, $\nabla\widehat{\phi}_h=0$ since $\widehat{\phi}_h=\mathcal{P}_h^0(\phi_h)\in \mathbb{P}_0(\T_h)$. Therefore, $J_1=0$.
\item In the case of $\alpha=2$, $\nabla\widehat{\phi}_h$ is a piecewise constant vector due to $\widehat{\phi}_h=\mathcal{P}_h^1(\phi_h)\in \mathbb{P}_1(\T_h)$. Then, by Green's formula and (\ref{p-fm}), we get
\begin{align*}
J_1 = \sum_{K\in\T_h}-\langle Q_0u-u, \nabla\widehat{\phi}_h\cdot\bn\rangle_{\partial K}=0.
\end{align*}
\end{itemize}
Thus, in both cases $\alpha=1$ and $\alpha=2$, we have
\begin{equation}\label{J1-est}
J_1=0.
\end{equation}
As to the second term {\color{red}$J_2$}, recalling the fact
\begin{align*}
\langle Q_n(\frac{\partial u}{\partial \bn})-\frac{\partial u}{\partial\bn}, \Delta \phi\rangle_{\partial\T_h}=0,
\end{align*}
we split {\color{red}$J_2$} into the following two terms:
\begin{align*}
J_2&=\langle Q_n(\frac{\partial u}{\partial \bn})-\frac{\partial u}{\partial\bn}, \widehat{\phi}_h\rangle_{\partial\T_h}\\
&=\langle Q_n(\frac{\partial u}{\partial \bn})-\frac{\partial u}{\partial\bn}, \mathcal{P}_h^{\alpha-1}(\Delta_w(Q_h\phi)-\Delta \phi)\rangle_{\partial\T_h}\nonumber\\
&\quad+\langle Q_n(\frac{\partial u}{\partial \bn})-\frac{\partial u}{\partial\bn}, \mathcal{P}_h^{\alpha-1}(\Delta \phi)-\Delta \phi\rangle_{\partial\T_h}.
\end{align*}
And then, by Cauchy-Schwarz inequality and (\ref{z2}) of Lemma \ref{l2} with $\gamma=k+1$, we get
\begin{align}\label{J2}
|J_2|&\leq (\sum_{K\in\T_h}h_K^{-1}\|Q_n(\frac{\partial u}{\partial \bn})-\frac{\partial u}{\partial\bn}\|_{L^2(\partial K)}^2)^{1/2}(\Theta_1^{1/2}+\Theta_2^{1/2})\nonumber\\
&\leq Ch^{k+1}|u|_{H^{k+3}(\Omega)}(\Theta_1^{1/2}+\Theta_2^{1/2}),
\end{align}
where
\begin{align*}
\Theta_1&:=\sum_{K\in\T_h}h_K\|\mathcal{P}_h^{\alpha-1}(\Delta_w(Q_h\phi)-\Delta \phi)\|_{L^2(\partial K)}^2,\\
\Theta_2&:=\sum_{K\in\T_h}h_K\|\mathcal{P}_h^{\alpha-1}(\Delta \phi)-\Delta \phi\|_{L^2(\partial K)}^2.
\end{align*}
From the trace inequality and the stability of $L^2$ projection $\mathcal{P}_h^{\alpha-1}$, it follows that
\begin{align*}
\Theta_1\leq C\sum_{K\in\T_h}\|\mathcal{P}_h^{\alpha-1}(\Delta_w(Q_h\phi)-\Delta \phi)\|_{L^2(K)}^2
\leq C\|\Delta_w(Q_h\phi)-\Delta \phi\|_{L^2(\T_h)}^2.
\end{align*}
Then, by the triangle inequality and (\ref{z3}) of Lemma \ref{l2}, we arrive at
\begin{align}\label{theta1}
\Theta_1&\leq C(\|\Delta_w(Q_h\phi)-\pi_h\Delta \phi\|_{L^2(\T_h)}^2+\|\pi_h\Delta\phi-\Delta \phi\|_{L^2(\T_h)}^2)\nonumber\\
&\leq Ch^{2\min\{k+1,\alpha\}}|\phi|_{H^{\alpha+2}(\Omega)}^2.
\end{align}

It follows from the trace inequality and the approximation property of $L^2$ projection $\mathcal{P}_h^{\alpha-1}$ that
\begin{align*}
\Theta_2&\leq \sum_{K\in\T_h}(h_K\|\mathcal{P}_h^{\alpha-1}(\Delta \phi)-\Delta \phi\|_{L^2(K)}^2+h_K |\mathcal{P}_h^{\alpha-1}(\Delta \phi)-\Delta \phi|_{H^1( K)}^2)\nonumber\\
&\leq Ch^{2\alpha}|\phi|_{H^{\alpha+2}(\Omega)}^2,
\end{align*}
which together with (\ref{theta1}) and (\ref{J2}) leads to
\begin{align}\label{J2-est}
|J_2|\leq Ch^{\min\{2k+2,k+1+\alpha\}}|u|_{H^{k+3}(\Omega)}|\phi|_{H^{\alpha+2}(\Omega)}.
\end{align}
Collecting (\ref{T2}), (\ref{J1-est}) and (\ref{J2-est}) yields
\begin{align*}
|T_2|\leq Ch^{\min\{2k+2,k+1+\alpha\}}|u|_{H^{k+3}(\Omega)}|\phi|_{H^{\alpha+2}(\Omega)},
\end{align*}
which combining with (\ref{T1-est}) and (\ref{l1-est}) completed the proof of (\ref{a1}).

As to the proof of (\ref{a2}), from the Cauchy-Schwarz inequality and Lemma \ref{l2} with $\gamma=k+1$  it follows that
\begin{align*}
|\ell_2(u, Q_h\phi)|
&=\left|\sum_{T\in\T_h} \langle \Delta u-\pi_h\Delta u, \frac{\partial}{\partial \bn}
(Q_0\phi)-Q_n(\frac{\partial \phi}{\partial\bn_e})\bn_e\cdot\bn\rangle_\pK\right|\nonumber\\
&\leq\left(\sum_{K\in\T_h} h_K\|\Delta u-\pi_h\Delta u\|_{L^2(\pK)}^2\right)^{1/2}\times\nonumber\\
&\quad  \left(\sum_{K\in\T_h} h^{-1}_K
\|\frac{\partial}{\partial\bn}
(Q_0\phi)-Q_n(\frac{\partial \phi}{\partial\bn})\|_{L^2(\pK)}^2\right)^{1/2}\nonumber\\
&\leq Ch^{k+1}|u|_{H^{k+3}(\Omega)}\cdot h^{\min\{k+1,\alpha\}}|\phi|_{H^{\alpha+2}(\Omega)}\nonumber\\
&\leq C h^{\min\{2k+2,k+1+\alpha\}}|u|_{H^{k+3}(\Omega)} |\phi|_{H^{\alpha+2}(\Omega)}.
\end{align*}
The proof is completed.
\end{proof}

\begin{theorem}\label{thm2}
Let $u_h=\{u_0, u_n\bm{n}_e\}\in V_h$ be the solution
of the SF-C0WG scheme (\ref{wg}). Assume that the exact solution $u\in H^{k+3}(\Omega)$ and the regularity assumption (\ref{reg}) holds true.
 Then, there exists a constant $C$ such that
\begin{equation}\label{L2err}
\|Q_0u-u_0\|_{L^2(\Omega)} \leq Ch^{k+3-\delta_{k,0}}|u|_{H^{k+3}(\Omega)}
\end{equation}
and
\begin{equation}\label{H1err}
\|\nabla(Q_0u-u_0)\|_{L^2(\Omega)} \leq Ch^{k+2}|u|_{H^{k+3}(\Omega)}.
\end{equation}
Here $\delta_{i,j}$ is the usual Kronecker's delta with value $1$
when $i=j$ and value $0$ otherwise.
\end{theorem}

\begin{proof}
Testing (\ref{dual}) by error function $e_0$ and then using a similar procedure as in the proof of the equation (\ref{mmmm}), we obtain
\begin{align}\label{tt}
(\chi, e_0)=(\Delta^2 \phi, e_0)_{\T_h}
%&=(\Delta_w (Q_hw), \Delta_w e_h)_{\T_h}
%- \langle (\nabla e_0-e_n\bn_e)\cdot\bn, \Delta w-\bbQ_h\Delta w %\rangle_{\partial\T_h}\nonumber\\
=\mathcal{A}_h(e_h, Q_h\phi)-\ell(\phi,e_h).
\end{align}
The error equation (\ref{ee}) gives
\begin{align*}
\mathcal{A}_h(e_h, Q_h\phi) = \ell(u,Q_h\phi),
\end{align*}
which combining with (\ref{tt}) leads to
\begin{align}\label{z0}
(\chi, e_0)=\ell(u,Q_h\phi)-\ell(\phi,e_h).
\end{align}

In view of Lemma \ref{l4}, we infer that
\begin{align}\label{z5}
|\ell(u, Q_h\phi)| \leq Ch^{\min\{2k+2,k+1+\alpha\}}|u|_{H^{k+3}(\Omega)}|\phi|_{H^{\alpha+2}(\Omega)}.
\end{align}

Using Lemma \ref{l3} with $\gamma = \alpha$ and  Theorem \ref{thm1}, we have
\begin{align*}
|\ell(\phi,e_h)|&\leq C h^{\min\{k+1,\alpha\}} |\phi|_{H^{\alpha+2}(\Omega)}\3bar e_h\3bar\\
&\leq Ch^{\min\{2k+2,k+1+\alpha\}}|u|_{H^{k+3}(\Omega)}|\phi|_{H^{\alpha+2}(\Omega)},
\end{align*}
which combining with (\ref{z0}) and (\ref{z4}) leads to
\begin{align}\label{z4}
|(\chi, e_0)| \leq C h^{\min\{2k+2,k+1+\alpha\}}|u|_{H^{k+3}(\Omega)}|\phi|_{H^{\alpha+2}(\Omega)}.
\end{align}

For the $L^2$-norm estimate of $e_0$, taking $\chi=e_0$ in the dual problem (\ref{dual})-(\ref{dual2}), and then using the estimate of (\ref{z4}) with the $H^4$-regularity, we find
\begin{align*}
\|e_0\|_{L^2(\Omega)}^2 \leq C h^{\min\{2k+2,k+3\}}|u|_{H^{k+3}(\Omega)}|\phi|_{H^{4}(\Omega)},
\end{align*}
which together with the assumption (\ref{reg}) with $\alpha=2$:
\[\|\phi\|_{H^4(\Omega)}\leq C\|e_0\|_{L^2(\Omega)}\]
completes the proof of (\ref{L2err}).

Then using the estimate of (\ref{z4}) with the $H^3$-regularity yields
\begin{align*}
|(\chi, e_0)| \leq C h^{k+2}|u|_{H^{k+3}(\Omega)}|\phi|_{H^3(\Omega)},
\end{align*}
which together with the assumption (\ref{reg}) with $\alpha=1$:
\[\|\phi\|_{H^3(\Omega)}\leq C\|\chi\|_{H^{-1}(\Omega)}\]
leads to
\begin{align}
\|\nabla e_0\|_{L^2(\Omega)} = \sup_{\chi\in H^{-1}(\Omega)}\frac{(\chi, e_0)}{\|\chi\|_{H^{-1}(\Omega)}}
\leq C h^{k+2}|u|_{H^{k+3}(\Omega)},
\end{align}
which ends the proof of (\ref{H1err}).
\end{proof}

By the triangle inequality, from Theorem \ref{thm2} and (\ref{Q0-approxi}), we immediately obtain the $L^2$ norm and $H^1$ norm error estimates between the exact solution $u$ and its WG finite element approximation $u_0$ as follows:
\begin{corollary}
Let $u_h=\{u_0, u_n\bm{n}_e\}\in V_h$ be the solution
of the SF-C0WG scheme (\ref{wg}). Assume that the exact solution $u\in H^{k+3}(\Omega)$ and the regularity assumption (\ref{reg}) holds true.
 Then, there exists a constant $C$ such that
\begin{equation}
\|u-u_0\|_{L^2(\Omega)} \leq Ch^{k+3-\delta_{k,0}}|u|_{H^{k+3}(\Omega)}
\end{equation}
and
\begin{equation}
\|\nabla(u-u_0)\|_{L^2(\Omega)} \leq Ch^{k+2}|u|_{H^{k+3}(\Omega)}.
\end{equation}
Here $\delta_{i,j}$ is the usual Kronecker's delta with value $1$
when $i=j$ and value $0$ otherwise.
\end{corollary}

\section{Numerical Experiments} \label{Sec:numeric}
In this section, we conduct some numerical experiments to verify the theoretical predication on the SF-C0WG method (\ref{wg}) and also to compare its numerical performance to the C0WG method (\ref{c0wg}) and the C0IP method (\ref{c0ipdg}).

\begin{example}
Consider the model problem (\ref{pde})-(\ref{pde-bc2}) with $\Omega=(0,1)^2$.
The source data $f$ and boundaries data $g_D$ and $g_N$ are chosen so that the exact solution is
\[u=\sin(\pi x)\sin(\pi y).\]
\end{example}
\begin{figure}[!h]
\centering
{\includegraphics[width=0.7\textwidth]{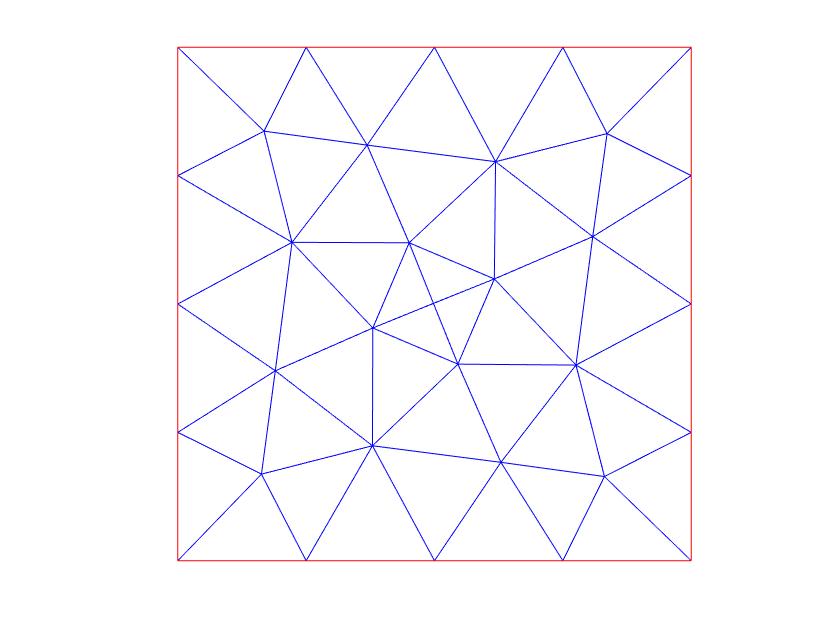}}
\caption{\label{grid1} The initial mesh. }
\label{fig:ex1-fig1}
\end{figure}

The initial mesh in our computation is shown in Figure \ref{fig:ex1-fig1}, which is generated by MATLAB function \texttt{initmesh}.
The next level of mesh is derived by uniformly refining the previous level of mesh.
The errors and the orders of convergence for the SF-C0WG method (\ref{wg}) with $k=0$ and $k=1$ are reported in Tables \ref{t1}, which confirm the theoretical predication in Theorem \ref{thm1} and Theorem \ref{thm2}.

Table \ref{t2} lists the errors and the rates of convergence for the C0WG method (\ref{c0wg}). The results in Table \ref{t1} and Table \ref{t2} show that both the SF-C0WG method and the C0WG method converge with the same rates, but the accuracy reached on a given mesh with a given polynomial degree is significant different. The SF-C0WG method is more accuracy than the C0WG method.

\begin{table}[h!]
  \centering \renewcommand{\arraystretch}{1.1}
  \caption{Error profiles and convergence rates of the SF-C0WG method.}\label{t1}
\begin{tabular}{c|c|cc|cc|cc}
\toprule[1.5pt]
$k$&level & $\3bar Q_hu-u_h\3bar$&Rate & $\|\nabla(u-u_0)\|$ &Rate  &$\|u-u_0\|$  &Rate    \\
\hline
 \multirow{5}*{0}
 &1&3.34E+00 & --&8.06E-02 & --&8.15E-03 & --\\
 &2&1.66E+00 & 1.0076&2.03E-02 & 1.9899&2.00E-03 & 2.0305\\
 &3&8.25E-01 & 1.0095&5.20E-03 & 1.9653&5.09E-04 & 1.9710\\
 &4&4.11E-01 & 1.0059&1.32E-03 & 1.9787&1.29E-04 & 1.9760\\
 &5&2.05E-01 & 1.0031&3.32E-04 & 1.9915&3.26E-05 & 1.9893\\
% &6&1.17e-01 & 0.9867&8.31E-05 & 1.9971&8.17E-06 & 1.9956\\
 \hline
 \multirow{5}*{1}
 &1&3.61E-01 & --&5.79E-03 & --&2.97E-04 & --\\
 &2&9.12E-02 & 1.9853&7.26E-04 & 2.9952&2.16E-05 & 3.7835\\
 &3&2.28E-02 & 1.9975&8.99E-05 & 3.0144&1.41E-06 & 3.9374\\
 &4&5.71E-03 & 2.0001&1.12E-05 & 3.0096&8.91E-08 & 3.9820\\
 &5&1.43E-03 & 2.0004&1.39E-06 & 3.0048&6.29E-09 & 3.8238\\
% &6&3.56e-04 & 2.0003&1.89E-07 & 2.8817&1.44E-08 & -1.1351\\
 \bottomrule[1.5pt]
\end{tabular}%
\end{table}%
%Table 2
\begin{table}[h!]
  \centering \renewcommand{\arraystretch}{1.1}
  \caption{Error profiles and convergence rates of the C0WG method.}\label{t2}
\begin{tabular}{c|c|cc|cc|cc}
\toprule[1.5pt]
$k$&level & $\3bar Q_hu-u_h\3bar$&Rate & $\|\nabla(u-u_0)\|$ &Rate  &$\|u-u_0\|$  &Rate    \\
\hline
 \multirow{5}*{0}
 &1&4.73E+00 & --&6.35E-01 & --&1.36E-01 & --\\
 &2&2.33E+00 & 1.0189&1.52E-01 & 2.0620&3.35E-02 & 2.0236\\
 &3&1.16E+00 & 1.0046&3.77E-02 & 2.0109&8.35E-03 & 2.0033\\
 &4&5.81E-01 & 1.0018&9.41E-03 & 2.0047&2.08E-03 & 2.0019\\
 &5&2.90E-01 & 1.0008&2.35E-03 & 2.0022&5.21E-04 & 2.0011\\
% &6&1.17e-01 & 0.9867&8.31E-05 & 1.9971&8.17E-06 & 1.9956\\
 \hline
 \multirow{5}*{1}
 &1&7.14E-01 & --&7.91E-02 & --&4.46E-03 & --\\
 &2&1.91E-01 & 1.9043&1.04E-02 & 2.9287&3.04E-04 & 3.8743\\
 &3&4.89E-02 & 1.9629&1.33E-03 & 2.9627&1.96E-05 & 3.9522\\
 &4&1.24E-02 & 1.9825&1.70E-04 & 2.9743&1.25E-06 & 3.9737\\
 &5&3.11E-03 & 1.9914&2.14E-05 & 2.9849&7.89E-08 & 3.9852\\
% &6&3.56e-04 & 2.0003&1.89E-07 & 2.8817&1.44E-08 & -1.1351\\
 \bottomrule[1.5pt]
\end{tabular}%
\end{table}%

%Table 2-2
\begin{table}[h!]
  \centering \renewcommand{\arraystretch}{1.1}
  \caption{Error profiles and convergence rates of the C0IP method.}\label{t2-2}
\begin{tabular}{c|c|cc|cc|cc}
\toprule[1.5pt]
$k$&level & $\| u-u_h\|_{dg}$&Rate & $\|\nabla(u-u_h)\|$ &Rate  &$\|u-u_h\|$  &Rate    \\
\hline
 \multirow{5}*{0}
 &1&1.33E+00 & --    &8.47E-02 & --    &1.02E-02 & --\\
 &2&6.40E-01 & 1.0595&2.24E-02 & 1.9150&2.91E-03 & 1.8035\\
 &3&3.20E-01 & 1.0009&5.92E-03 & 1.9222&7.99E-04 & 1.8635\\
 &4&1.60E-01 & 0.9955&1.52E-03 & 1.9584&2.10E-04 & 1.9310\\
 &5&8.03E-02 & 0.9983&3.86E-04 & 1.9824&5.35E-05 & 1.9689\\
 \hline
 \multirow{5}*{1}
 &1&2.00E-01 & --    &6.24E-03 & --    &3.47E-04 & --\\
 &2&5.15E-02 & 1.9533&7.83E-04 & 2.9950&2.58E-05 & 3.7489\\
 &3&1.31E-02 & 1.9780&9.62E-05 & 3.0238&1.70E-06 & 3.9194\\
 &4&3.30E-03 & 1.9870&1.19E-05 & 3.0157&1.09E-07 & 3.9712\\
 &5&8.29E-04 & 1.9930&1.48E-06 & 3.0076&6.53E-09 & 4.0564\\
 \bottomrule[1.5pt]
\end{tabular}%
\end{table}%
Table \ref{t2-2} shows the errors and the rates of convergence for the C0IP method (\ref{c0ipdg}). The errors in the first column of Table \ref{t2-2} is measured in the following $H^2$-like norm tailored for the C0IP method:
\[
 \|v\|_{dg}:=\left[\sum_{K\in\T_h}|v|_{H^2(K)}^2+\sum_{e\in\E_h}h_e^{-1}\|\ljump\nabla v\rjump\|_{L^2(e)}^2 \right]^{1/2}.
\]
The results in Table \ref{t1} and Table \ref{t2-2} show that both the SF-C0WG method and the C0IP method converge with the same rate and the accuracies are also similar when the errors are measured in $H^1$ semi-norm and $L^2$ norm.

%Table 3
\begin{table}[h!]
  \centering \renewcommand{\arraystretch}{1.1}
  \caption{Comparison of assembling time and solving time for the C0WG method  and the SF-C0WG method.}\label{t3}
\begin{tabular}{c|c|cc|cc}
\toprule[1.5pt]
\multirow{2}*{$k$}&\multirow{2}*{level}&\multicolumn{2}{c|}{C0WG method }&\multicolumn{2}{c}{SF-C0WG method}\\
\cline{3-6}
& & Assembling Time &Solving Time & Assembling Time &Solving Time     \\
\hline
 \multirow{5}*{0}
 &1&0.052233 & 0.001690&0.065710 & 0.003116\\
 &2&0.175486 & 0.007218&0.157510 & 0.006127\\
 &3&0.734831 & 0.039118&0.546378 & 0.030866\\
 &4&2.602549 & 0.171534&2.240196 & 0.133192\\
 &5&10.67890 & 0.874419&8.766250 & 0.628217\\
 \hline
 \multirow{5}*{1}
 &1&0.347160 & 0.027630&0.057160 & 0.016028\\
 &2&0.184210 & 0.020082&0.201938 & 0.017132\\
 &3&0.864096 & 0.072827&0.793079 & 0.061735\\
 &4&3.537430 & 0.458761&2.800155 & 0.305041\\
 &5&23.91767& 2.630822&12.14147 & 1.752295\\
 \bottomrule[1.5pt]
\end{tabular}%
\end{table}
A comparison of the assembling time and solving time for both the C0WG method and the SF-C0WG method is displayed in Table \ref{t3}. It can be observed that the assembling time and solving time for the SF-C0WG method is always smaller than that for the C0WG method.

%Table 4
\begin{table}[h!]
  \centering \renewcommand{\arraystretch}{1.1}
  \caption{Comparison of assembling, solving and total time for the C0IP method  and the SF-C0WG method.}\label{t4}
\begin{tabular}{c|c|cc|cc}
\toprule[1.5pt]
\multirow{2}*{Time (sec.)}&\multirow{2}*{level}&\multicolumn{2}{c|}{$k=0$ }&\multicolumn{2}{c}{$k=1$}\\
\cline{3-6}
& & {\color{red}C0IP}  &SF-C0WG  & {\color{red}C0IP} &SF-C0WG     \\
\hline
 \multirow{5}*{Assemble}
 &1&0.073452 & 0.065710&0.091383 & 0.057160\\
 &2&0.163071 & 0.157510&0.252789 & 0.201938\\
 &3&0.711536 & 0.546378&1.004696 & 0.793079\\
 &4&2.308666 & 2.240196&3.720932 & 2.800155\\
 &5&9.169572 & 8.766250&14.90928 & 12.14147\\
 \hline
 \multirow{5}*{Solve}
 &1&0.007031 & 0.003116&0.009475 & 0.016028\\
 &2&0.004608 & 0.006127&0.015663 & 0.017132\\
 &3&0.021312 & 0.030866&0.060428 & 0.061735\\
 &4&0.079957 & 0.133192&0.252089 & 0.305041\\
 &5&0.385537 & 0.628217&1.610523 & 1.752295\\
 \hline
 \multirow{5}*{Total}
 &1&0.080483 & 0.068825&0.100858 & 0.073188\\
 &2&0.167680 & 0.163636&0.268452 & 0.219071\\
 &3&0.732848 & 0.577244&1.065125 & 0.854814\\
 &4&2.388623 & 2.373388&3.973021 & 3.105196\\
 &5&9.555110 & 9.394467&16.51981 & 13.89377\\
 \bottomrule[1.5pt]
\end{tabular}%
\end{table}
The assembling time, solving time, and total time (the sum of the assembling and solving time) for both the C0IP method and the SF-C0WG method are illustrated in Table \ref{t4}. As can be seen, although the solving time of C0IP method is less than the SF-C0WG method, the assembling time and total time for the SF-C0WG method is always smaller than that for the C0IP method.

\section{Appendix} In this section, we shall introduce some technique tools which are useful in the $L^2$ and $H^1$ norm error analysis.

In order to prove Lemma \ref{lem:p3}, we introduce the following two lemmas.
%lem:p1
\begin{lemma}\label{lem:p1}
For any $K\in\T_h$, there holds
\begin{align}\label{pp}
\Delta_w(Q_hw) = \Delta w,\quad \forall w\in \mathbb{P}_{k+2}(K).
\end{align}
\end{lemma}
\begin{proof}
For any $w\in \mathbb{P}_{k+2}(K)$, from the definitions of $Q_0$ and $Q_n$, we have $Q_0w =w$ and $Q_n(\frac{\partial w}{\partial\bn_e})=\frac{\partial w}{\partial\bn_e}$.  Then, for any $K\in\T_h$ and $v\in\mathbb{P}_{k+3}(K)$, from the definition (\ref{wl}) of the weak laplacian, it follows that
\begin{align*}
(\Delta_wQ_hw, v)_K &= -(\nabla Q_0w, \nabla v)_K + \langle Q_n(\frac{\partial w}{\partial\bn_e})\bn_e\cdot\bn, v\rangle_{\partial K}\\
&= -(\nabla w, \nabla v)_K + \langle \frac{\partial w}{\partial\bn}, v\rangle_{\partial K}\\
&=(\Delta w, v)_K,
\end{align*}
which completes the proof.
\end{proof}

%lem:app2
Let $\mathcal{P}_h^{k+2}: L^2(\T_h)\rightarrow \mathbb{P}_{k+2}(\T_h)$ be the element-wise defined $L^2$ orthogonal projection.
\begin{lemma}\label{lem:p2}
Assume $\phi\in H^{\alpha+2}(\Omega)$ with $\alpha=1,2$. Then, there holds
\begin{align}\label{p2}
\|\Delta_wQ_h(\phi-\mathcal{P}_h^{k+2}\phi)\|_{L^2(\T_h)}
\leq Ch^{\min\{k+1,\alpha\}}|\phi|_{H^{\alpha+2}(\Omega)}.
\end{align}
\end{lemma}
\begin{proof}
For simplicity, denote by $w=\phi-\mathcal{P}_h^{k+2}\phi$. It follows from (\ref{key}) and the Cauchy-Schwarz inequality that
\begin{align}\label{aa2}
(\Delta_wQ_hw, v)_{\T_h} &= (\Delta Q_0w, v)_{\T_h}
+\langle Q_n(\frac{\partial w}{\partial\bn})-\frac{\partial }{\partial\bn}(Q_0w), v\rangle_{\partial {\T_h}}\nonumber\\
&\leq \|\Delta Q_0w\|_{L^2(\T_h)}\|v\|_{L^2(\T_h)}\nonumber\\
&\quad+(\sum_{K\in\T_h}h_K^{-1}\|Q_n(\frac{\partial w}{\partial\bn})-\frac{\partial }{\partial\bn}(Q_0w)\|_{\partial K}^2)^{1/2}
(\sum_{K\in\T_h}h_K\|v\|_{\partial K}^2)^{1/2}\nonumber\\
&\leq C(\|\Delta Q_0w\|_{L^2(\T_h)}+(\sum_{K\in\T_h}h_K^{-1}\|Q_n(\frac{\partial w}{\partial\bn})-\frac{\partial }{\partial\bn}(Q_0w)\|_{\partial K}^2)^{1/2})\|v\|_{L^2(\T_h)},
\end{align}
for any $v\in\mathbb{P}_{k+3}(\T_h)$.

Letting $v=\Delta_wQ_hw$ in (\ref{aa2}), and then cancelling out $\|\Delta_wQ_hw\|_{L^2(\T_h)}$ from both sides yields
\begin{align}\label{aa3}
\|\Delta_wQ_hw\|_{\T_h}\leq C[\|\Delta Q_0w\|_{L^2(\T_h)}+(\sum_{K\in\T_h}h_K^{-1}\|Q_n(\frac{\partial w}{\partial\bn})-\frac{\partial }{\partial\bn}(Q_0w)\|_{\partial K}^2)^{1/2}].
\end{align}

Since the interpolant $Q_0$ preserves polynomials of degree up to $k+2$, it is easy to know
\[
  Q_0(\mathcal{P}_h^{k+2}\phi)=\mathcal{P}_h^{k+2}\phi.
\]
Then, by the triangle inequality, we have
\begin{align}\label{aa4}
\|\Delta Q_0w\|_{L^2(\T_h)}
&=\|\Delta Q_0(\phi-\mathcal{P}_h^{k+2}\phi)\|_{L^2(\T_h)}\nonumber\\
&\leq \|\Delta (Q_0\phi-\phi)\|_{L^2(\T_h)}+\|\Delta(\phi-\mathcal{P}_h^{k+2}\phi)\|_{L^2(\T_h)}\nonumber\\
&\leq Ch^{\min\{k+1,\alpha\}}|\phi|_{H^{\alpha+2}(\Omega)}.
\end{align}
Since $Q_n$ and $Q_0$ preserve the polynomials of order $k+1$ and $k+2$ respectively, there holds
\begin{align*}
Q_n(\frac{\partial w}{\partial\bn})-\frac{\partial }{\partial\bn}(Q_0w)
&=Q_n(\frac{\partial }{\partial\bn}(\phi-\mathcal{P}_h^{k+2}\phi))-
\frac{\partial }{\partial\bn}(Q_0(\phi-\mathcal{P}_h^{k+2}\phi))\\
&=Q_n(\frac{\partial \phi}{\partial\bn})-
\frac{\partial }{\partial\bn}(Q_0\phi),
\end{align*}
which together with (\ref{z2}) of Lemma \ref{l2} leads to
\begin{align}\label{aa5}
\sum_{K\in\T_h}h_K^{-1}\|Q_n(\frac{\partial w}{\partial\bn})-\frac{\partial }{\partial\bn}(Q_0w)\|_{\partial K}^2
&=\sum_{K\in\T_h}h_K^{-1}\|Q_n(\frac{\partial \phi}{\partial\bn})-
\frac{\partial }{\partial\bn}(Q_0\phi)\|_{\partial K}^2\nonumber\\
&\leq Ch^{2\min\{k+1,\alpha\}}|\phi|_{H^{\alpha+2}(\Omega)}^2.
\end{align}
Combining the estimates of (\ref{aa3}), (\ref{aa4}) and (\ref{aa5}) completes the proof of (\ref{p2}).
\end{proof}

%%lem:app3
%\begin{lemma}\label{lem:p3}
%Let $\phi\in H^{\alpha+2}(\Omega)$ with $\alpha=1,2$. Then, there holds
%\begin{align}\label{p3}
%|\Delta_w(Q_h\phi)|_{H^{\alpha}(\T_h)}\leq Ch^{\min\{k+1-\alpha,0\}}|\phi|_{H^{\alpha+2}(\Omega)}.
%\end{align}
%\end{lemma}

Now, we are ready to give the proof of Lemma \ref{lem:p3} below.
\begin{proof}
 In view of (\ref{pp}) of Lemma \ref{lem:p1}, we have
\[\Delta_w Q_h(\mathcal{P}_h^{k+2}\phi)=\Delta(\mathcal{P}_h^{k+2}\phi)\]
on each element $K$ of $\T_h$.

If $\alpha>k$, we have $|\Delta(\mathcal{P}_h^{k+2}\phi)|_{H^{\alpha}(\T_h)}=0$ since $\Delta(\mathcal{P}_h^{k+2}\phi)\in\mathbb{P}_k(\T_h)$. Therefore, by the triangle inequality, we have
\begin{align}\label{aa1}
|\Delta_w(Q_h\phi)|_{H^{\alpha}(\T_h)}
&=|\Delta_wQ_h(\phi-\mathcal{P}_h^{k+2}\phi)+\Delta(\mathcal{P}_h^{k+2}\phi)|_{H^{\alpha}(\T_h)}\nonumber\\
&\leq|\Delta_wQ_h(\phi-\mathcal{P}_h^{k+2}\phi)|_{H^{\alpha}(\T_h)}
+|\Delta(\mathcal{P}_h^{k+2}\phi)|_{H^{\alpha}(\T_h)}\nonumber\\
&=|\Delta_wQ_h(\phi-\mathcal{P}_h^{k+2}\phi)|_{H^{\alpha}(\T_h)}.
\end{align}
Then, from the inverse inequality, (\ref{aa1}) and (\ref{p2}) of Lemma \ref{lem:p2}, it follows that
\begin{align*}%\label{aa6}
|\Delta_w(Q_h\phi)|_{H^{\alpha}(\T_h)}
&\leq Ch^{-\alpha}\|\Delta_wQ_h(\phi-\mathcal{P}_h^{k+2}\phi)\|_{L^2(\T_h)}\nonumber\\
&\leq Ch^{\min\{k+1-\alpha,0\}}|\phi|_{H^{\alpha+2}(\Omega)}.
\end{align*}

If $\alpha\leq k$, from the triangle inequality, the inverse inequality and (\ref{p2}) of Lemma \ref{lem:p2}, we can infer that
\begin{align*}
&|\Delta_w(Q_h\phi)|_{H^{\alpha}(\T_h)}\nonumber\\
&\leq |\Delta_wQ_h(\phi-\mathcal{P}_h^{k+2}\phi)|_{H^{\alpha}(\T_h)}
+|\Delta(\phi-\mathcal{P}_h^{k+2}\phi)|_{H^{\alpha}(\T_h)}
+|\Delta\phi|_{H^{\alpha}(\T_h)}
\nonumber\\
&\leq Ch^{-\alpha}\|\Delta_wQ_h(\phi-\mathcal{P}_h^{k+2}\phi)\|_{L^2(\T_h)}
+Ch^{\min\{k+1-\alpha,0\}}|\phi|_{H^{\alpha+2}(\T_h)}\nonumber\\
&\leq Ch^{\min\{k+1-\alpha,0\}}|\phi|_{H^{\alpha+2}(\T_h)}.
\end{align*}
Therefore, in all cases, we have
\begin{align*}
|\Delta_w(Q_h\phi)|_{H^{\alpha}(\T_h)}
\leq Ch^{\min\{k+1-\alpha,0\}}|\phi|_{H^{\alpha+2}(\T_h)},
\end{align*}
as desired.
\end{proof}

\end{document}